\documentclass[a4paper]{amsart}
\usepackage{amscd,amsthm,amsfonts,amssymb,amsmath,esint}
\usepackage{amsmath,tikz-cd}
\usepackage{fullpage}
\usepackage[all]{xy}
\usepackage{mathtools}
\usepackage{cite}
\usepackage[colorlinks=false,
linkcolor=red,       
anchorcolor=red,  
citecolor=blue,        
]{hyperref}

\newcommand{\BC}{{\mathbb {C}}} 
 \newcommand{\BF}{{\mathbb {F}}}
\newcommand{\BG}{{\mathbb {G}}}

 \newcommand{\BP}{{\mathbb {P}}}
\newcommand{\BQ}{{\mathbb {Q}}} 
\newcommand{\BS}{{\mathbb {S}}}

 \newcommand{\BZ}{{\mathbb {Z}}}

\newcommand{\CA}{{\mathcal {A}}} 
\newcommand{\CC}{{\mathcal {C}}} \renewcommand{\CD}{{\mathcal {D}}}

\newcommand{\CK}{{\mathcal {K}}} \newcommand{\CL}{{\mathcal {L}}}
 \newcommand{\CN}{{\mathcal {N}}}
\newcommand{\CO}{{\mathcal {O}}}

\newcommand{\CU}{{\mathcal {U}}} 
 
\newcommand{\CY}{{\mathcal {Y}}}

\newcommand{\fa}{{\mathfrak{a}}}

\newcommand{\fg}{{\mathfrak{g}}} \newcommand{\fh}{{\mathfrak{h}}}

 \newcommand{\fp}{{\mathfrak{p}}}

\theoremstyle{plain}
\newtheorem{thm}{Theorem}[section] \newtheorem{cor}[thm]{Corollary}
\newtheorem{lem}[thm]{Lemma}

 \newtheorem{prop}[thm]{Proposition}

\newtheorem{defn}[thm]{Definition}
 \newtheorem{lem-defn}[thm]{Lemma-Definition}

\theoremstyle{remark} \newtheorem{remark}[thm]{Remark}
\theoremstyle{remark} 
\theoremstyle{remark} \newtheorem{example}{Example}

\numberwithin{equation}{section}

\everymath{\displaystyle}
\author{Yingqi Liu}
\title{A characterization of irreducible Hermitian symmetric spaces of tube type by $\BC^{*}$-actions}
\address{Yingqi Liu,  AMSS, Chinese Academy of Sciences, 55 ZhongGuanCun East Road, Beijing, 100190, China and  University of Chinese Academy of Sciences, Beijing, China}
\email{liuyingqi@amss.ac.cn}
\begin{document}
	\date{\today}

	\maketitle
	
	\begin{abstract}
		A $\BC^{*}$-action on a projective variety $X$  is said to be of Euler type at a nonsingular fixed point $x$ if the isotropy action of $\BC^{*}$ on $T_{x}X$ is by scalar multiplication. In this paper, it's proven that a smooth projective variety  of Picard number one $X$ is isomorphic to an irreducible Hermitian symmetric space of tube type if and only if  for a general pair of points $x,y$ on $X$, there exists a $\BC^{*}$-action on $X$ which is of Euler type at $x$ and its inverse action is of Euler type at $y$.
	\end{abstract}
	\section{Introduction}
	\subsection{Main result}
	The study of  complex torus actions on algebraic varieties is a classical topic in algebraic geometry. It is  an interesting problem to classify algebraic varieties with special $\BC^{*}$-actions. 
	Let $X$ be a smooth projective variety. A $\BC^{*}$-action on $X$ is said to be equalized at a fixed point $x$ if any weight of the isotropy action on $T_{x}X$ equals to 0 or $\pm 1$. We call the action equalized if it's equalized at each fixed point.  Denote by $X^{\BC^{*}}$ the set of fixed points of the $\BC^{*}$-action. An irreducible component of $X^{\BC^{*}}$ is called extremal if it intersects with general $\BC^{*}$-orbit closures.
	In the series of works  \cite{adjunction_c*}\cite{small_bandwidth}\cite{ORSW_2021}, the authors study equalized $\BC^{*}$-actions on projective manifolds with isolated extremal fixed components. For rational homogenous spaces they proved the following:
	\begin{thm}\label{small_bandwidth}
		\cite{ORSW_2021}  Let $X=G/P$ be a rational homogenous space of Picard number one, then $X$ admits an equalized $\BC^{*}$-action with two isolated extremal fixed points if and only if X  is isomorphic to one of the followings:\par
		(i) a smooth hyperquadric $\BQ^{n}$.\par 
		(ii) the Grassmaninan variety $Gr(n,2n)$. \par 
		(iii) the Lagrangian Grassmanian variety $Lag(n,2n)$. \par
		(iv) the spinor variety $\BS_{2n}$. \par 
		(v) the 27 dimensional  $E_{7}$-variety $E_{7}/P_{7}$. 
	\end{thm}
	
	The varieties classified above are exactly  irreducible Hermitian symmetric spaces (IHSS for short) of tube type. Recall that an IHSS is called of tube type if its dual, as a bounded symmetric domain, is holomorphically equivalent to a tube domain over a self dual cone.  
	In \cite{Mok_}  Mok showed that  an IHSS is of tube type if and only if for a point $o \in X$ the complement of the  vectors of maximal rank in $T_{o}(X)$  is a hypersurface. Translating in the language of $\BC^{*}$-actions, one can regard  Theorem 1.1 as a generalization of Mok's result to rational homogenous spaces.\par 
	It is then a natural problem to characterize IHSS of tube type by $\BC^{*}$-actions  in a more general context. A $\BC^{*}$-action on a projective manifold $X$ is said to be of Euler type at a fixed point $x$ if the isotropy action of $\BC^{*}$ on $T_{x}X$ is by scalar multiplication, or equivalently if the $\BC^{*}$-action is equalized at $x$ and  $x$ is an isolated extremal fixed point.  In Theorem 1.1, by taking certain conjugates of $\BC^{*}$ in $Aut^{0}(X)$  one can show that for a general pair of points $x,y$ on $X$ there exists a $\BC^{*}$-action which is of Euler type at $x$ and its inverse action is of Euler type at $y$ (see Section 4 for more details). Here a general  pair of points on $X$ is defined as a pair of points  lying in a Zariski open dense subset of $X \times X$. Our main result proves the converse:
	
	\begin{thm}\label{main_thm}
		Let $X$ be a smooth projective variety of Picard number one, then $X$ is isomorphic to an IHSS of tube type if and only if for a general pair of points $x,y$ on $X$, there exists a $\BC^{*}$-action on $X$ which is of Euler type at $x$ and its inverse action is of Euler type at $y$.
	\end{thm}

	\subsection{Outline of the proof}

	The main ingredient of the proof is the VMRT theory developed by Hwang and Mok.   Let $X$ be a Fano manifold of Picard number one and let $\CK$ be a fixed irreducible dominant family of minimal rational curves on $X$. The variety of minimal rational tangents (VMRT for short) at a general point $x \in X$ is the closed subvariety $\CC_{x} \subseteq \BP T_{x}X$ consisting of all tangent directions at $x$ of curves in $\CK$ passing through $x$. A large part of the global geometry of the manifold is controlled by the VMRT $\CC_{x} \subseteq \BP T_{x}X$ at a general point $x$. The first step is to show that $X$ in Theorem 1.2 is an equivariant compactification of vector group hence $X$ can be recovered from its VMRT by Cartan-Fubini type theorem  \cite[Theorem 1.2]{Cartan_fubini}.  
	Then we follow the methods developed in \cite{prog_02} \cite{nonzero_prog} to classify the projective subvariety $\CC_{x} \subseteq \BP T_{x}X$ by studying its prolongaion of infinitesimal linear automorphisms. 
	\begin{defn}
		(1) Let $\mathfrak{g} \subseteq \mathfrak{g}l(V)$ be a Lie subalgebra, then the k-th prolongation of $\mathfrak{g}$ is the space of symmetric multi-linear homomorphisms $A: Sym^{k+1}V \rightarrow V$ such that for any fixed $v_{1},...,v_{k} \in V$, the homomorphism $A_{v_{1},...,v_{k}}:V \rightarrow V$ defined by:
		\begin{equation*}
			v \in V \rightarrow A(v,v_{1},...,v_{k}) \in V
		\end{equation*}	
		is in $\mathfrak{g}$.\par 
		(2) Let $S \subseteq \BP V$ be a projective subvariety. Let $\hat{S} \subseteq V$ be its affine cone and $T_{\alpha}(\hat{S})$ the tangent space at a smooth point $\alpha \in \hat{S}$. The Lie algebra of infinitesimal linear automorphisms of $\hat{S}$ is
		\begin{equation*}
			\mathfrak{a}ut(\hat{S}) =\{ g \in End(V)| g(\alpha) \in T_{\alpha}(\hat{S}) \, \text{for any smooth point $\alpha \in \hat{S} $} \}=\{g \in End(V) |  exp(tg) \cdot \hat{S} \subset \hat{S}, \, t \in \BC \}.
		\end{equation*}
		Its k-th prolongation $\mathfrak{a}ut(\hat{S} )^{(k)}$ will be called the k-th prolongation of $S \subseteq \BP V$.
	\end{defn}
	A crucial step of our proof is the following result.
	\begin{thm}\label{non_vmrt_}
		Let $X$ be a smooth projective variety of Picard number one. Assume for a general pair of points $(x,y) \in X \times X$ , there is a $\BC^{*}$-action on $X$ which is of Euler type at  $x$ and its inverse action is of Euler type at $y$, then for the VMRT $\CC_{x} \subseteq \BP T_{x}X$ we have:
		\begin{equation}
			dim(\mathfrak{a}ut(\hat{\CC_{x}})^{(1)})=dim(X)
		\end{equation}
	\end{thm}
	The identity (1.1) is known to be valid for any IHSS. Let $X=G/P$ be an IHSS defined by a semi-simple algebraic group $G$ and a maximal parabolic subgroup $P$. Denote by $P^{-}$ the opposite group of $P$ and let $x=eP$, $y=\dot{w}_{0}x$ where $w_{0}$ is the longest element in the Weyl group. Then there is an equalized $\BC^{*}$-action on $X$ such that $x$ is an isolated extremal fixed point. And $\fa ut( \hat{\CC_{x}} )^{(1)} $ can be  identified with the Lie algebra of $R_{u}(P)$,  induced by the $\BZ / 2 \BZ$-grading of the Lie algebra $\mathfrak{g}=Lie(G)$. When X is of tube type, then $y$ is also an isolated extremal fixed point of the $\BC^{*}$-action and X is an equivariant compactification of $R_{u}(P)$ with open orbit $R_{u}(P) \cdot y$. Thus in this case the nonzero prolongation of  $\fa ut(\hat{\CC_{x}})$ comes from the vector group compactification on $X$ with origin $y$.\par 
	Generally assume that $X$ satisfies the conditions of Theorem 1.2. Then the vector groups action of $T_{y}X$ and $T_{x}X$ on themselves can be extended to  equivariant compactifications of vector groups on $X$.  Choosing a suitable non-degenrate projective embedding of $X \subseteq \BP V$. The actions of $T_{y}X$ and $T_{x}X$ on $X$ can be lifted to linear actions on $V$. Identifying  $Lie(T_{y}X),Lie(T_{x}X)$ with their images in $\fg l(V)$,  the adjoint actions of the two Lie subalgebras inside $\fg l(V)$ will induce the identification: $Lie(T_{y}X) \cong \mathfrak{a}ut(\hat{\CC_{x}})^{(1)}$.  \par 
	In \cite{nonzero_prog} Fu and Hwang classified irreducible non-degenrate nonsingular projective subvariety $S \subseteq \BP V$ with nonzero prolongations. Most of them are the VMRT of an  IHSS or a nonsingular linear section of some IHSS.  
	Using a case by case  calculation of $dim(\mathfrak{a}ut(\hat{S})^{(1)})$, we show that $\CC_{x} \subseteq \BP T_{x}X$ is projectively isomorphic to the VMRT of an IHSS. Then by a $\BC^{*}$-equivariant  Cartan-Fubini extension theorem, we conclude that $X$ is $\BC^{*}$-isomorphic to an IHSS. The $\BC^{*}$-action on IHSS was studied in detail in \cite{ORSW_2021}, as a corollary it has two isolated extremal fixed points if and only if it is isomorphic to an IHSS of tube type.\par 
	The article is organized as follows. In Section 2 we first review  the approach  in \cite{euler_sym} to associate $\BC^{*}$-actions with vector group actions, then we apply it to the case when the $\BC^{*}$-action has two isolated extremal fixed points. In Section 3 we study the prolongations of  projective varieties, we first prove Theorem 1.4  then we calculate the dimension of prolongation for certain projective subvariety with nonzero prolongations. In Section 4 we finish the proof of our main result.
	\subsection*{Notations}
	Throughout this article we work over the field of complex numbers. Given a line bundle $\CL$ on a variety $X$, the principal open subset of a section $s \in H^{0}(X,\CL)$ is denoted by $D_{+}(s)=\{x \in X: s(x) \not=0\}$, and the cycle-theoretic zero locus of $s$ is denoted by $Z(s)$. For a vector space $V$ of dimension $n$, we identify the regular functions on $V$ with the total space of symmetric multilinear forms:
	$\BC[V] \cong \oplus_{k \geqslant 0} Sym^{k}(V^{*})$ by assigning a function $f$ on $V$ to $P_{f}=\sum_{k \geqslant 0} P_{f,k}$ such that $f(v)=\sum_{k \geqslant 0}P_{f,k}(v,v...,v)$.
	\section{From $\BC^{*}$-action to vector group action}
	First we recall some basic facts and notions on $\BC^{*}$-actions.
	Let $X$ be a smooth projective variety with a $\BC^{*}$-action.  
	For a nontrival $\BC^{*}$-orbit of a point $z$ on $X$, the orbit map $\psi_{z}:\BC^{*} \rightarrow X$ extends uniquely to a morphism $\Psi_{z}:\BP^{1} \rightarrow X$. We denote the source (resp. sink) of the orbit to  be $\lim_{t \rightarrow 0}tz=\Psi_{z}(0)$ (resp. $\lim_{t \rightarrow \infty} tz=\Psi_{z}(\infty)$). Denote by $\CY$ the set of irreducible components of $X^{\BC^{*}}$. Then for each $Y \in \CY$, $Y$ is a smooth closed subvariety. The isotropy action of $\BC^{*}$ on $TX|_{Y}$ gives a decomposition of $TX|_{Y}=T^{+}(Y) \oplus T^{-}(Y) \oplus TY$,  where $T^{+}(Y),T^{-}(Y)$ are the subbundles of $TX|_{Y}$ on which $\BC^{*}$ acts with positive, negative weights. Denote $C^{\pm}(Y)=\{ x \in X: \lim_{t^{\pm 1} \rightarrow 0} tx  \in Y \}$. We recall the following theorem of  Bia\l ynicki-Birula  \cite{BBthm}.
	\begin{thm} Assume that $X$ is a smooth projective variety with a $\BC^{*}$-action, then:\par 
		(1) For each $Y \in \CY$, $C^{\pm}(Y)$ are locally closed subsets and there are decompositions: 
		\begin{equation*}
			X=\bigcup\limits_{Y \in \CY} C^{+}(Y)=\bigcup\limits_{Y \in \CY} C^{-}(Y)
		\end{equation*}
		(2)  For each $Y \in \CY$, there are $\BC^{*}$-isomorphisms $C^{+}(Y) \cong T^{+}(Y)$ and $C^{-}(Y) \cong T^{-}(Y)$ lifting the natural map $C^{\pm}(Y) \rightarrow Y$.The map $C^{\pm}(Y) \rightarrow Y$ is algebraic and is a $\BC^{v^{\pm}(Y)}$-fibration, where we set $v^{\pm}(Y)=rank(T^{\pm}(Y))$. \par 
	\end{thm}
	We call the unique $Y$ such that $C^{+}(Y)$ (resp. $C^{-}(Y)$) is dense  the source (resp. sink) of the action.  For any two components $Y,Y' \in \CY$, we call $Y \prec Y'$ if there exists a point $x \in X$ such that $lim_{t \rightarrow 0} t \cdot x \in Y$ and $lim_{t \rightarrow \infty} t \cdot x \in Y'$. 
	\begin{prop}\cite[Lemma 2.8 ]{small_bandwidth} and \cite[Lemma 3.5]{adjunction_c*}\\
		Let $X$ be a smooth projective variety of Picard number one. Assume that there is a $\BC^{*}$-action on $X$ which is of Euler-type at  $x$, then there is a unique $Y \in \CY$ such that
		\begin{equation}
			\{Y' \in \CY:  Y' \prec Y   \}=\{x\}.
		\end{equation}
		Moreover $C^{-}(Y)$ is a line bundle over $Y$, i.e., $v^{-}(Y)=1$.
	\end{prop}

	\par 
	To associate $\BC^{*}$-actions with vector group actions we recall  Euler-symmetric varieties defined by Fu and Hwang in \cite{euler_sym}.
	\begin{defn}
		Let $Z \subseteq \BP V$ be a projective subvariety. A $\BC^{*}$-action on $Z$ is called of Euler-type at a nonsingular point $x$ if the isotropic action on the tangent space $T_{x}Z$ is by scalar multiplication. We say $Z \subseteq \BP V$ is an Euler-symmetric variety if for a general point $x \in Z$, there exists a $\BC^{*}$-action which is of Euler type at x, where the $\BC^{*}$-action comes from a multiplicative subgroup of $GL(V)$.
	\end{defn}
	From \cite[Theorem 3.7]{euler_sym} any Euler-symmetric variety is an equivariant compactification of vector group. Moreover we have the following by \cite[Theorem 3.7]{euler_sym}.
	\begin{prop}
		Assume $X$ to be a normal projective variety and  let $\CL$ be a very ample line bundle on $X$, then the followings are equivalent:\par 
		(1) For a general point $x$ on $X$, there is a $\BC^{*}$-action which of Euler type at $x$.\par 
		(2) $X$ is an equivariant compactification of vector group and the scalar multiplication of $\BC^{*}$ on the vector group can be extended to a $\BC^{*}$-action on $X$. \par 
		(3) The projective subvariety $X \subset \BP H^{0}(X,\CL)^{\vee}$ is Euler-symmetric.
	\end{prop}
	
In the following we always assume that $X$ is a smooth projective variety of Picard number one. We study $\BC^{*}$-actions on $X$ from the  perspective of local projective differential geometry following  \cite{proj_geo_homogenous} and \citen{euler_sym},
\begin{defn}\label{fund_syst}
	(1)	Let $x \in X \subseteq  \BP V$ be a nonsingular point of a nondegenerate projective variety, and let $\CL=\CO_{\BP V}(1)|_{X}$ be the line bundle on $X$.
	For each nonnegative integer $k$, $\mathfrak{m}_{x,X}^{k}$ be the $k$-th power of the maximal ideal $\mathfrak{m}_{x,X}$. For a section $s \in H^{0}(X,\CL)$, let $j_{x}^{k}(s)$ be the $k$-jet of $s$ at $x$ such that $j_{x}^{0}=s_{x} \in \CL_{x}$.  The induced homomorphism :
	\begin{equation*}
		(V^{*} \cap Ker(j_{x}^{k-1}))/(V^{*} \cap Ker(j_{x}^{k})) \rightarrow \CL_{x} \otimes Sym^{k}T_{x}^{*}X
	\end{equation*}
	is injective. For each $k \geqslant 2$, the subspace $\BF_{x}^{k} \subseteq Sym^{k}T_{x}^{*}X$ defined by the image of this homomorphism is called the $k$-th fundamental form of  $X$ at $x$. Set $F_{x}^{0}=\BC$ and $F_{x}^{1}=T_{x}^{*}X$, The collection of subspaces $\BF_{x}=\oplus_{k \geqslant 0} \BF_{x}^{k} \subset \oplus_{k \geqslant 0} Sym^{k}T_{x}^{*}X$ is called the system of fundamental forms of $X$ at $x$.\par 
	(2)  Let $W$ be a vector space. For $w \in W$, the contraction homomorphism  $\iota_{w}: Sym^{k+1}W^{*} \rightarrow Sym^{k}W^{*}$ sending $\phi  \in Sym^{k+1}W^{*}$ to $\iota_{w}\phi \in Sym^{k}W^{*}$ is defined by:
	\begin{equation*}
		\iota_{w}\phi(w_{1},...,w_{k})=\phi(w,w_{1},...,w_{k})
	\end{equation*}
	for any $w_{1},...,w_{k} \in W$. By convention we define $\iota_{w}(Sym^{0}W^{*})=0$. \par 
	(3) A subspace $\BF=\oplus_{k \geqslant 0} F^{k} \subset \oplus_{k \geqslant 0} Sym^{k}W^{*}$ with $F^{0}=\BC, F^{1}=W^{*}, F^{r} \not =0$, and $F^{r+i}=0$ for all $i \geqslant 1$ is called a $\text{symbol system of rank r}$ if $\iota_{w} F^{k+1} \subseteq F^{k}$ for any $w \in W$ and any $k \geqslant 0$.
\end{defn}
We recall the following theorem of Cartan (see for example \cite[Section 2.1]{proj_geo_homogenous}).
\begin{thm} 
	Let $X \subseteq \BP V$ be a nondegenerate subvariety, and  let $x \in X$ be a general point. Then the system of fundamental forms  $\BF_{x} \subset \oplus_{k \geqslant 0} Sym^{k}T_{x}^{*}X$  is a symbol system.
\end{thm}
Now assume that $X$ admits a $\BC^{*}$-action which is of Euler type at a fixed point $x$. Denote by $\CL$ a very ample line bundle on $X$ and $V=H^{0}(X,\CL)^{\vee}$. 	Choose a $\BC^{*}$-linearization on $\CL$ such that if we write $H^{0}(X,\CL)=\oplus_{k=0}^{r} H^{0}(X,\CL)_{w_{k}}$ as the sum of nonzero weight subspaces with respected to the $\BC^{*}$-action, then we have $0=w_{0}> w_{1}  > w_{2} > ... > w_{r}$. The fundamental forms of $X \subset \BP V$ at $x$ can be identified as follows.
\begin{lem}\label{lem}
	(1)  The subspace $\{ s \in H^{0}(X,\CL) \, | \, D_{+}(s)=C^{+}(x)  \}$ is of dimension one  and  it equals to $H^{0}(X,\CL)_{0}$.	 Take a nonzero section $s_{0}$ with $D_{+}(s_{0})=C^{+}(x)$ and  consider the linear map:
	\begin{align*}
		\eta: H^{0}(X, \CL) &\rightarrow \BC[C^{+}(x)]  \\
		s &\rightarrow \eta(s) 
	\end{align*}
	defined by: $s|_{C^{+}(x)}=\eta(s)s_{0}|_{C^{+}(x)}$, for any $s \in H^{0}(X, \CL) $. Then:\par 
	(2) $\eta$ is an injective $\BC^{*}$-equivariant linear map, where the $\BC^{*}$-action on $\BC[C^{+}(x)]$ is given by $z \cdot f (u)=f(z^{-1} \cdot u)$, for any $z \in \BC^{*}$, $u \in C^{+}(x)$ and $f \in \BC[C^{+}(x)]$.\par 
	(3) Identifying $\BC[C^{+}(x)] \cong \BC[T_{x}X] \cong \oplus_{k \geqslant 0} Sym^{k}T_{x}^{*}X$, then the image of $\eta$ is identified with $\BF_{x}$, under which $\eta(H^{0}(X,\CL)_{w_{k}})=\BF_{x}^{-w_{k}}$. Furthermore $w_{1}=-1$ and $\eta|_{H^{0}(X,\CL)_{w_{1}}}$ is an isomorphism. \par
\end{lem}
\begin{proof}
	As $C^{+}(x)$ is isomorphic to the affine space $T_{x}X$,  $D_{x} := X \backslash C^{+}(x)$ is a divisor by Hartgo's theorem and  $Cl(X)$ is freely generated by the irreducible components  of $D_{x}$. Then by our assumption that $X$ is of Picard number one,  $D_{x}$ is  the prime generator of $Cl(X)$.
	Thus we can write $\CL \cong \CO_{X}(r_{0}D_{x})$ for some positive integer $r_{0}$. For any nonzero section $s$  with $D_{+}(s)=C^{+}(x)$, write $Z(s)=rD_{x}$ for some positive integer $r$. Then we have $rD_{x} \sim r_{0}D_{x}$, implying that  $r=r_{0}$ and thus the subspace $\{ s \in H^{0}(X,\CL): D_{+}(s)=B^{+}(x)  \}$ is of dimension one. The subspace is $\BC^{*}$-invariant and we denote its weight to be $w'$.  
	Take a nonzero section $s_{0}$ in this subspace and define $\eta$ as above, then $\eta$  is injective as $C^{+}(x)$ is open. For any $s \in H^{0}(X,\CL)$, $z \in \BC^{*}$ and $u \in C^{+}(x)$ we have:
	\begin{equation*}
		(z \cdot s)(u)=z \cdot s (z^{-1} \cdot  u)=z \cdot (\eta(s)(z^{-1} \cdot  u)s_{0}(z^{-1} \cdot u))=\eta(s)(z^{-1}u) \, (z\cdot s_{0})(u)=z^{w'}(z \cdot \eta(s))(u) s_{0}(u),
	\end{equation*}
	implying that $\eta(z \cdot s)=z^{w'}(z \cdot \eta(s))$.  For any $ s \in H^{0}(X,\CL)_{w_{k}}$, viewing $\eta(s)$ as a regular function on $T_{x}X$ via  $C^{+}(x) \cong T_{x}X$, then we have:
	\begin{align*}
		\eta(z \cdot s)(u)=z^{w_{k}}\eta(s)(u),\,\,
		z \cdot  \eta(s)(u)=\eta(s)(z^{-1} \cdot u)=\eta(s)(z^{-1}u)
	\end{align*}
	for any $u \in T_{x}X$. Combined with $\eta(z \cdot s)=z^{w'}(z \cdot \eta(s))$ this shows that $\eta(s)$ is a homogenous polynomial on $T_{x}X$ with $deg(\eta(s))=w'-w_{k} \, \geqslant 0$. This implies that $w'=w_{0}=0$ and  $deg(\eta(s))=-w_{k}$. Thus $\eta$ is $\BC^{*}$-equivariant,  and for any nonzero section $s \in H^{0}(X,\CL)_{0}$, $\eta(s)$ is a constant, which shows that $C^{+}(x)=D_{+}(s)$,  proving (1) and (2). \par  
	For a section $s \in H^{0}(X,\CL)$,write $\eta(s) =\sum_{k \geqslant 0} \eta_{k}(s)$ as the sum of homogenous functions.  The $k$-th jet of $s$  at $x$ equals the class of $\eta(s)|_{x}$ in $ \CO_{x}/\mathfrak{m}_{x}^{k}$, whence under the identification it is represented by $\sum_{i=0}^{k-1} \eta_{i}(s)$.  On the other hand we have proved that  $\eta$ maps $H^{0}(X,\CL)_{w_{k}}$ into the space of homogenous polynomials of degree $-w_{k}$. Thus by the definition of fundamental forms we have $\eta(H^{0}(X,\CL)_{w_{k}})=\BF_{x}^{-w_{k}}$ and Im$(\eta)=\BF_{x}$. Finally  as $\CL$ is very ample, by \cite[Proposition 7.2]{Ha} Im$(\eta)$ generates $\BC[C^{+}(x)]$ as a $\BC-$algebra. Thus  Im$(\eta)$ contains the space of linear functions $T_{x}^{*}X$, which has to be the image of $H^{0}(X,\CL)_{w_{1}}$. It then follows that $w_{1}=-1$ and  $\eta|_{H^{0}(X,\CL)_{w_{1}}}$ is an isomorphism.
\end{proof}

From Lemma 2.7, the fundamental form at $x$ is given by the collection of linear injection $\eta|_{H^{0}(X,\CL)_{w_{k}}}$ for each $k$.
Dually  it is given by the collection of  surjective linear map 
$	\Pi_{k} : \, Sym^{-w_{k}}(T_{x}X) \rightarrow H^{0}(X,\CL)_{w_{k}}^{\vee} $ for each $k$,
where  $\Pi_{0}: \BC \rightarrow H^{0}(X,\CL)_{0}^{\vee}$ maps 1 to the unique  $e_{0} \in H^{0}(X,\CL)_{0}^{\vee}$ such that $e_{0}(s_{0})=1$ and $\Pi_{1}$ is an isomorphism by Lemma 2.7 (3). 
Denote by $f: X \rightarrow \BP H^{0}(X,\CL)^{\vee}$ the projective embedding. Then under the identification $C^{+}(x) \cong T_{x}X$,  we have	$f(u)=[e_{0} + \sum_{k=1}^{r} \Pi_{k}(u,...,u)]$ for each $u \in T_{x}X$ and $f(x)=f(0)=[e_{0}]$. At this point we can deduce an easy corollary:
\begin{cor}
	Let $X,\CL,x$ be as above, if $r=1$ then $X$ is  isomorphic to a projective space. 	
\end{cor}
\begin{proof}
	Write $V=H^{0}(X,\CL)^{\vee}$ and $W=H^{0}(X,\CL)^{\vee}_{w_{1}}$. Then we have $V=\BC e_{0} \oplus W$ and  $C^{+}(x)=f(T_{x}X)=\{[e_{0} + w]: w \in W \}$ as $\Pi_{1}$ is surjective. Thus the sink of the action equals $\BP W$ and hence $f$ is an isomorphism onto $\BP V$.
\end{proof}
Now assume that $x$ is chosen as a general point.
The contraction homomorphism in Definition \ref{fund_syst}(2) defines   a locally nilpotent action of $T_{x}X$ on  $\oplus_{k \geqslant 0} Sym^{k}T_{x}^{*}X$.  Then by Theorem 2.6, $\BF_{x}$ is  a finite dimensional $T_{x}X$-invariant subspace. Dually it defines a nilpotent action of $T_{x}X$ on $H^{0}(X,\CL)^{\vee}$  through each $\Pi_{k}$. More precisely for each $0 \leqslant  k   \leqslant r $ and for any  $v \in T_{x}X$, denote  $\Gamma_{v}|_{H^{0}(X,\CL)_{w_{k}}^{\vee}}:  H^{0}(X,\CL)_{w_{k}}^{\vee}  \rightarrow H^{0}(X,\CL)_{w_{k+1}}^{\vee} $  by:
\begin{equation}
	\Gamma_{v} \circ  \Pi_{k}(v_{1},...,v_{k})=\Pi_{k+1}(v,v_{1},...,v_{k}) , \,\,\,\,  \text{$\forall k \geqslant 1$} \text{\,\,\,and\,\,\,}
	\Gamma_{v}(e_{0})=\Pi_{1}(v), \Gamma_{v}(H^{0}(X,\CL)_{w_{r}}^{\vee})=0.
\end{equation}
Then $f|_{T_{x}X}$ can be also written as: $f(w)=\sum_{k=0}^{r} \Gamma_{w}^{k}(e_{0})$ for any $w \in T_{x}X$.\par 
In the following we denote by  $V_{k}=H^{0}(X, \CL)_{w_{k}}^{\vee}$ for each $k$.  It can be easily checked that the induced linear map $\Gamma: T_{x}X \rightarrow \mathfrak{g}l(V)$  is  a homomorphism between Lie algebras.  As $x$ is a general point,  $X \subset \BP V$ is Euler-symmetric by Proposition 2.4, where the vector group action on $X$ is given by a linear representation $\rho_{x}: T_{x}X \rightarrow GL(V)$ defined as the following (see \cite[Proof of Theorem 3.7]{euler_sym}):

\begin{equation}
	\rho_{x}(u)(v)=\sum_{l=0}^{r} \sum_{k=0}^{l}  \tbinom{l}{k} \Gamma_{u}^{l-k} (v_{k}),
\end{equation}
for any $u \in T_{x}X$ and $v=\sum_{k=0}^{r}v_{k} \in V$, where $v_{k} \in V_{k}$ for each $k$.

\begin{prop}\label{x_act}
	Let $X,\CL, \rho_{x},f$ be as above, where $x$ is chosen as a general point. Then: \par 
	(1) For each $0 \leqslant  k \leqslant r$, $w_{k}=-k$. \par 
	(2) For each $k \geqslant 1$ and for any $u_{1},...,u_{k} \in T_{x}X$, $\Pi_{k}(u_{1},...,u_{k})=\Gamma_{u_{1}} \circ ... \circ \Gamma_{u_{k}}(e_{0})$.\par  
	(3) The induced action of $T_{x}X$ on $\BP V$ leaves $X$ invariant and lifts the action of $T_{x}X$ on $C^{+}(x)$. Moreover denote by $d\rho_{x}: T_{x}X \rightarrow  \mathfrak{g}l(V)$ the differential map of $\rho_{x}$, then $d\rho_{x}|_{V_{k}}=(k+1) \, \Gamma|_{V_{k}}$ for each $k$. In particular we have
	\begin{equation*}
		d\rho_{x}(u) \cdot V_{k} \subseteq V_{k+1},
	\end{equation*}
	for any $u \in T_{x}X$, any $k  \geqslant 0$ and $d \rho_{x}(u) (e_{0})=\Pi_{1}(u)=\Gamma_{u}(e_{0})$.
\end{prop}
\begin{proof}
	As $x$ is a general point, $\BF_{x}$ is a symbol system. Thus  for any $l$ such that $\BF_{x}^{l+1} \not =0$  we have $\BF_{x}^{l} \not= 0$ as well. Then from Lemma 2.7 (3), we conclude that   $\BF_{x}=\oplus_{k=0}^{r} \BF_{x}^{r}$, $w_{k}=-k$ and  $\eta(H^{0}(X,\CL)_{w_{k}})=\BF_{x}^{k}$ for each $k$. (2)  is a direct consequence of our definition in (2.2). For (3) to show the first assertion, it suffices to check $\rho_{x}(u') \cdot f(u)=f(u+u')$, for any $u,u' \in T_{x}X $.
	For any $u,u' \in T_{x}X$, we have:
	\begin{equation*}
		f(u+u')=	[\sum_{k=0}^{r} \Gamma_{u+u'}^{k}(e_{0})]=[\sum_{l=0}^{r} \sum_{k=0}^{l}  \tbinom{l}{k}   \Gamma_{u}^{k} \circ  \Gamma_{u'}^{l-k}(e_{0})]=[\rho_{x}(u')(\sum_{k=0}^{r} \Gamma_{u}^{k}(e_{0}))]=\rho_{x}(u') \cdot f(u).
	\end{equation*}
	Now for any  $u \in T_{x}X$ and $v=\sum_{k=0}^{r} v_{k} \in V$:
	\begin{align*}
		d\rho_{x}(u)(v)&=\frac{d}{d_{z}}_{|z=0}(\rho_{x}(zu)(v))=\frac{d}{d_{z}}_{|z=0}(\sum_{l=0}^{r} \sum_{k=0}^{l} \tbinom{l}{k}z^{l-k}\Gamma_{u}^{l-k}(v_{k})) =\sum_{l=1}^{r} l  \cdot  \Gamma_{u}(v_{l-1}).
	\end{align*}
	In other words we have $d\rho_{x}|_{V_{k}}=(k+1) \,  \Gamma|_{V_{k}}$, for any $k \geqslant 0$. In particular, $d\rho_{x}.V_{k} \subseteq V_{k+1}$ and $d\rho_{x}(u)(e_{0})=\Gamma_{u}(e_{0})=\Pi_{1}(u)$ for any $u \in T_{x}X$
\end{proof}
Next we apply it to the case when both the  sink and the source are isolated. Assume that for a general pair of points $x,y$ on $X$, there is a $\BC^{*}$-action on $X$ which is of Euler type at the source $x$ and its inverse action is of Euler type at $y$.  Take the linearization of the inverse action on $\CL$ such that the decomposition of the associated weight subspaces equals $H^{0}(X,\CL)=\oplus_{k=0}^{r} H^{0}(X,\CL)'_{w'_{k}}$, where $0=w_{0}' >...>w_{r}'$. Applying Proposition 2.9, we have the following:
\begin{cor}\label{y_act}
	Let $X,\CL , x$ and $y$ be as above, then:\par 
	(1) We have $H^{0}(X,\CL)_{w_{k}}=H^{0}(X,\CL)'_{w'_{r-k}}$ and  $w_{k}=-r-w'_{r-k}$.\par 
	(2) We have $V_{r}=H^{0}(X,\CL)'_{w_{0}'}=\{s \in H^{0}(X,\CL): D_{+}(s)=X^{-}(y) \}$ and $dim(H^{0}(X,\CL)'_{w_{0}'})=1$. Furthermore $f(y)=[e_{r}]$ where $e_{r}$ is a nonzero element in $V_{r}$.\par 
	(3) There is a  vector group action of $T_{y}X$ on $V$ namely $\rho_{y}: T_{y}X \rightarrow   GL(V)$, such that the induced action on $\BP V$ leaves $X$ invariant and lifts the action of $T_{y}X$ on $C^{-}(y)$. Moreover for the differential map $d\rho_{y}$ we have $d\rho_{y} \cdot  V_{0}=0$ and:
	\begin{equation*}
		d\rho_{y}(w) \cdot V_{k+1} \subseteq V_{k},
	\end{equation*}
	for any $w \in T_{y}X$ and for  any $k  \geqslant 0$.
\end{cor}
\begin{proof}
	(1) follows directly from the definition of $w_{k}'$. (2) follows by (1), Lemma 2.7 (1) and Proposition 2.9 (3).  For (3) we dually denote by $V_{k}'=(H^{0}(X,\CL)'_{w_{k}'})^{\vee}$ for each $k$.  Then by (1) we have $V_{k}'=V_{r-k}$. And  from Proposition 2.9 (3)  there is a vector group action of $T_{y}X$ on $V$, such that the induced action on $X$ extends the action of $T_{y}X$ on $C^{-}(y)$. Moreover $d\rho_{y} \cdot V_{k}' \subseteq V_{k+1}'$, equivalently we have $d\rho_{y} \cdot V_{k+1} \subseteq V_{k}$ for any $k \geqslant 0$ and $d\rho_{y} \cdot V_{0}=0$. 
\end{proof}
\section{Prolongations of projective subvarieties}
In this section,  we study prolongation of projective subvarieties in two steps following \cite{nonzero_prog}\cite{prog_02}. Firstly we study the prolongation of the VMRT of Euler-symmetric varieties at a general point. Then we calculate the dimension of $\mathfrak{a}ut(\hat{S})^{(1)}$  for certain $S \subset \BP V$ with nonzero prolongation, which was  explicitly formulated in \cite{nonzero_prog}. \par 

\subsection{Prolongation of the VMRT}
In this section we apply the results built in Section 2 to prove Theorem \ref{non_vmrt_}. Throughout the section, $X$ is assumed to be a Fano manifold of Picard number one. Let us recall the definition of VMRT.
\begin{defn} An irreducible component $\CK$  of the space Ratcurves$^{n}(X)$ of rational curves on $X$ is called a minimal rational component if the subvariety $\CK_{x}$ of $\CK$ parameterizing curves passing through a general point $x \in X$ is non-empty and proper. Curves parameterized by $\CK$ will be called minimal rational curves. Let $\rho: \CU \rightarrow \CK$ be the universal family and $\mu: \CU \rightarrow X$ the evaluation map. The tangent map $\tau: \CU \dashrightarrow \BP T(X)$ is defined by $\tau(u)=T_{\mu(u)}(\mu(\rho^{-1}(\rho(u))))$. The closure $\CC \subset \BP T(X)$  of its image is the total space of variety of minimal rational tangents. The natural projection $\CC \rightarrow X$ is a proper surjective morphism and a general fiber $\CC_{x} \subset \BP T_{x}X$ is called the variety of minimal rational tangents (VMRT for short) at the point $x \in X$.
\end{defn}
The VMRT of an Euler-symmetric variety at a general point is related to its fundamental forms as follows.
\begin{defn}
	Assume that for a general point $x$ on $X$, there exists a $\BC^{*}$-action on $X$ which is of Euler type at $x$. Take a very ample line bundle $\CL$ on $X$ and denote by $f: X \rightarrow \BP H^{0}(X,\CL)^{\vee}$ the projective embedding. Denote by $\BF_{x}$ the fundamental forms of $X$ at $x$ and take $\eta$ as in Lemma 2.7.  For each $k \geqslant 2$ denote:
	\begin{align*}
		Bs(\BF_{x}^{k})&=\{[w] \in \BP T_{x}X: \phi(w,...,w)=0, \forall \phi \in \BF_{x}^{k} \subset Sym^{k}T_{x}^{*}X \}
	\end{align*}
	Then $Bs(\BF_{x}^{k})=\{[w] \in \BP T_{x}X: \eta(s)(w)=0,  \forall s \in H^{0}(X,\CL)_{w_{k}} \}$ and we have the inclusions: $Bs(\BF_{x}^{2}) \subset Bs(\BF_{x}^{3})  \subset ... \subset   Bs(\BF_{x}^{r})$ as $\BF_{x}$ is a symbol system .  Denote the base locus of fundamental forms at $x$ to be $Bs(\BF_{x})=Bs(\BF_{x}^{l_{0}})$, where $l_{0}$ is the smallest integer $l$ such that $Bs(\BF_{x}^{l})$ is non-empty.
\end{defn}
\begin{prop}\label{vmrt_nond}\cite[Prop 4.4 and Prop 5.4(iii)]{euler_sym}\\
	Let $X,x,\CL$ be as above. Let $\CK$ be the family of minimal rational curves on $X$ and $\CC \subset \BP T(X)$ the VMRT-structure on $X$. Then $\CC_{x}=Bs(\BF_{x}) \subset \BP T_{x}X$, which is an irreducible, nonsingular and non-degenerate projective subvariety.
\end{prop}
Next we aim to study the prolongation of infinitesimal linear automorphism of $C_{x} \subseteq \BP T_{x} X$, the most important example of which is the case of  IHSS.
\begin{example}\label{ihss}
	Let $G$ be a simple algebraic group, $B \subset G$ a Borel subgroup, and $T \subset B$ a maximal torus. Denote  by $\Phi $ the root system, $\Delta \subset \Phi$ the simple roots and $\CD$ the Dynkin diagram. Let $\fg=Lie(G)$ be the Lie algebra of $G$ and let $\fh \subset \fg$ be the Cartan subalgebra. For a subset $I \subset \Delta$, denote by $P_{I}$ the standard parabolic subgroup indexed by $I$, and denote by $P_{I}^{-} \subset G$ the opposite group of $P_{I}$. We write the quotient $G/P_{I}$ as $\CD(I)$, which is a smooth  projective rational variety of Picard number $|I|$. In the following we always assume $I=\{\alpha\}$ for a simple root $\alpha \in \Delta$.
	For any root $\beta \in \Phi$, the multiplicity of the simple root  $\alpha$ in $\beta$ is denoted by $m_{\alpha}(\beta)$. Denote by $y=eP$ and $x=\dot{w_{0}}P$, where $w_{0}$ is the longest element in the Weyl group of $G$.\par 
	Write $\fg=\mathop{\oplus}\limits_{k \in \BZ }\fg_{k}$, where $\fg_{k}=\mathop{\oplus}\limits_{m_{\alpha}(\beta)=k}\fg_{\beta}$  for each $k \in \BZ$. Then $X=G/P_{\alpha}$ is called an IHSS if  $\fg$ admits a short grading: $\fg=\fg_{-1} \oplus \fg_{0} \oplus \fg_{1}$, or equivalently $R_{u}(P_{\alpha}^{-})$ and $R_{u}(P_{\alpha})$ are vector groups (see for example \citen{equi_flag}). In this case, $G/P_{\alpha}$ is an equivariant compactification of the vector group $R_{u}(P_{\alpha}^{-})$, with origin $y$. Define a $\BC^{*}$-action on $\CD(I)$ by the cocharcter $\sigma_{\alpha}: \BG_{m} \rightarrow T$, such that $\sigma_{\alpha}(\beta)=\delta_{\alpha,\beta}, \forall \beta \in \Delta $. The $\BC^{*}$-action is equalized at the sink $y$ and  $C^{-}(y)=R_{u}(P^{-}) \cdot y$. This shows that $\CD(I)$ is Euler-symmetric by Proposition 2.4. Denote by $H \subset P_{\alpha}$ the Levi-part of $P$, then $y$ is $H$-fixed. And $\fg_{0}=Lie(H),\fg_{1}=Lie(R_{u}(P_{\alpha}))$, $\fg_{-1}=Lie(R_{u}(P_{\alpha}^{-}))$. Finally the prolongation of $\fa ut (\hat{\CC_{y}})$ can be interpreted as adjoint actions inside the simple grading algebra $\fg$ as follows.\par 
	(1)   $T_{y}X=\fg/\fp \cong \fg_{-1}$ and $\CC_{y} \subset \BP T_{y}X$ is the unique closed orbit of the isotropy action of $H$ on $\BP T_{y}X\cong \BP \fg_{-1}$.\par 
	(2)  $\fa ut(\hat{\CC_{y}}) \cong \fg_{0}$ is given by the adjoint action of $\fg_{0}$ on $\fg_{-1}$. \par
	(3)  $\fa ut(\hat{\CC_{x}})^{(1)} \cong \fg_{1}$ is given by the adjoint action:
	\begin{align}
		\fg_{1} &\rightarrow Hom(\fg_{-1},\fg_{0})\\
		\alpha  &\rightarrow (\beta \rightarrow  [\alpha,\beta]), \nonumber
	\end{align}
	where the image lies in $Sym^{2}T_{x}^{*}X\otimes T_{x}(X)$ as $\fg_{-1}$ is abelian. We list IHSS and their VMRTs in the following table.
	\begin{table}[h]
		\begin{center}
			\caption{IHSS and their VMRTs}
			\begin{tabular}{|c|c|c|c|c|c|c|} \hline 
				\text{IHSS} $X=G/P$ & $\BQ^{n}$ & $Gr(a,a+b)$ & $\BS_{n}$  & $ Lag(n,2n)$ & $E_{6}/P_{1}$ & $E_{7}/P_{7}$ \\
				\hline
				\text{VMRT $\CC_{x}$} & $\BQ^{n-2}$ &   $\BP^{a-1} \times \BP^{b-1}$   & $Gr(2,n)$ & $\BP^{n-1}$ & $\BS_{5}$ & $E_{6}/P_{1}$ \\ \hline 
				$\CC_{x} \subset \BP T_{x}(X)$ & Hyperquadric & Segre  &  Pl$\ddot{u}$cker  & second Veronese & Spinor & Severi\\  \hline
				
			\end{tabular}
		\end{center}
	\end{table}
\end{example}
Now back to the general situation, assume that for a general pair of points $x,y$ on $X$, there is a $\BC^{*}$-action which is of Euler type at $x$ and its inverse action is of Euler type at $y$. Take a very ample line bundle $\CL$ on $X$ and denote  $V=H^{0}(X,\CL)^{\vee}$. Let $f: X \rightarrow \BP V$ be the projective embedding. Recall $\rho_{x}, \rho_{y}$ the linear actions of $T_{x}X$ and $T_{y}X$ on $V$ respectively. We denote by $\fg_{1}=\text{I}m(d\rho_{x}) \subset \fg l(V)$ and $\fg_{-1}=\text{I}m(d\rho_{y}) \subset \fg l(V)$. By definition $\fg_{1},\fg_{-1} \subset \fa ut(\hat{X})$, thus  $\forall \alpha \in \fg_{1}, \forall \beta \in \fg_{-1}$, $\gamma := [\alpha,\beta]$ lies in $ \fa ut(\hat{X})$ as well. Denote by $G_{\gamma} := \{ exp(z\gamma)\, | \,  z \in \BC \} \subset GL(V)$ the one-parameter subgroup. Then the action of $G_{\gamma}$ on $\BP V$ leaves $X$ invariant. Moreover we have:
\begin{lem}
	$x$ and $y$ are fixed by $G_{\gamma}$. Denote by $\Phi_{\gamma}: G_{\gamma} \rightarrow GL(T_{x}X) $ the  induced isotropy action of $G_{\gamma}$ on $T_{x}X$, then for any  $ w \in T_{x}X$ we have:
	\begin{equation}
		\Pi_{1}((d\Phi_{\gamma})(\gamma)(w))=[\gamma,d\phi_{x}(w)]  \cdot e_{0} \in V_{1}.
	\end{equation}
	\begin{proof}
		By Proposition \ref{x_act}(iii) and Corollary \ref{y_act}(3), $\alpha \cdot V_{k} \subset V_{k+1},$ and $ \beta \cdot V_{k+1} \subset V_{k}$ for each $0 \leqslant k \leqslant r-1$. Thus $\gamma \cdot V_{k} \subset V_{k}$ for any $0 \leqslant k \leqslant r$, whence  $G_{\gamma} \cdot V_{k} \subset V_{k}$ for each $k$. This particularly shows that $x$ and $y$ are fixed by $G_{\gamma}$ as $x=[e_{0}],y=[e_{r}]$ and $V_{0},V_{r}$ are both of dimension one. It also implies that the action of $G_{\gamma}$  commutes with the $\BC^{*}$-action.
		Now for any nonzero tangent vector $w \in T_{x}X$, under the identification $C^{+}(x) \cong T_{x}X$, consider the holomorphic arc $\theta_{w}$ on X through $x$:
		\begin{align*}
			\theta_{w}: \BC &\rightarrow X \subset \BP V\\
			z & \rightarrow f(zw)=[\sum_{k=0}^{r} z^{k} \Gamma_{w}^{k}(e_{0})] 
		\end{align*}
		Denote the closure of its image to be $C_{w}$. Then $C_{w}$ is a non-trival $\BC^{*}$-orbit closure  with  source $x$. As the action of $G_{\gamma}$ commutes with the $\BC^{*}$-action,  $g \cdot C_{w}$ is also a non-trival $\BC^{*}$-orbit closure through $g \cdot x=x$ for any $g \in G_{\gamma}$. This implies that $g \cdot f(w) \in C^{+}(x)$ for any $w$, in other words we have $ g \cdot C^{+}(x) \subset C^{+}(x)$. \par 
		Now we calculate the action of $G_{\gamma}$ on $T_{x}X$ via $T_{x}X \cong C^{+}(x)$. Write $\gamma \cdot  e_{0}=c e_{0}$ for some $c \in \BC$, then for any $g_{t}=exp(t \gamma)$ we have:
		\begin{align*}
			f(g_{t} \cdot w)&=[\sum_{k=0}^{r} \Gamma_{g_{t} \cdot w}^{k}(e_{0})]=[e_{0}+\Gamma_{g_{t}.w}^{1}(e_{0})+\sum_{k=2}^{r} \Gamma_{g_{t}.w}^{k}(e_{0})],\\
			g_{t} \cdot f(w)&=g_{t} \cdot [\sum_{k=0}^{r} \Gamma_{w}^{k}(e_{0})]=[e_{0}+e^{-tc}g_{t} \cdot \Gamma_{w}^{1}(e_{0})+\sum_{k=2}^{r}e^{-tc} g_{t} \cdot \Gamma_{w}^{k}(e_{0})].
		\end{align*}
		Thus from $f(g_{t}.w)=g_{t} \cdot f(w)$ we conclude that:
		\begin{equation}
			\Pi_{1}( g_{t} \cdot w)=\Gamma_{g_{t} \cdot w}^{1}(e_{0})=exp(t \, \gamma|_{V_{1}}-tc \, \,  \text{Id}|_{V_{1}} ) \cdot \Pi_{1}(w),
		\end{equation}
		from which we see that the action of $G_{\gamma}$ on $T_{x}X$ via $T_{x}X \cong C^{+}(x)$ is linear. Then as the differential map of  the isomorphism $T_{x}X \cong C^{+}(x)$ at $0 \in T_{x}X$ equals I$d|_{T_{x}X}$, we conclude that $\Phi_{\gamma}(g)(w)=g \cdot w$.  Thus:
		\begin{align*}
			\Pi_{1}(d\Phi_{\gamma)}(\gamma)(w))&=\frac{d}{dt}_{|t=0}(\Pi_{1}(g_{t} \cdot w))=\frac{d}{dt}_{|t=0}(e^{-tc}e^{t\gamma} \cdot \Pi_{1}(w))\\
			&=-c \, \Pi_{1}(w)+\gamma.\Pi_{1}(w)
			=[\gamma,d \phi_{x}(w)].e_{0}
		\end{align*}
		from $\Pi_{1}(w)=\Gamma_{w}(e_{0})=d\phi_{x}(w) \cdot e_{0}$.
	\end{proof}
\end{lem} 
As $G_{\gamma}$ fixes $x$, it acts on the family of minimal rational curves through $x$. Thus the image of $\Phi_{\gamma}$ is contained in $Aut^{0}(\hat{\CC_{x}})$ and  $d\Phi_{\gamma}(\gamma) \in \fa ut(\hat{\CC_{x}})$. Then under the identification $\Pi_{1}$, we can rewrite Lemma 3.4 as follows, which is a generalization of the map (3.1).
\begin{cor}
	Consider the linear map $\lambda:	\fg_{-1} \rightarrow  Sym^{2}(V_{1}^{*}) \otimes V_{1} $ given by:
	\begin{align}
		\lambda(d\phi_{y}(\beta)): \,\,\,\,\, \,\,\,\, V_{1} \times V_{1} &\longrightarrow V_{1}\\ \nonumber
		(\Pi_{1}(\alpha),\Pi_{1}(\xi)) &\longrightarrow [[d\phi_{y}(\beta),d\phi_{x}(\alpha)],d\phi_{x}(\xi)].e_{0},
	\end{align}
	for any $\beta \in T_{y}X$ and for any  $\alpha, \xi \in T_{x}X$.
	Then Im$(\lambda) \subset \fa ut(\hat{\CC_{x}})^{(1)}$, under the identiication $\Pi_{1}$.
\end{cor}
\begin{proof}
	It suffices to check $\lambda(d\phi_{y}(\alpha)$ is symmetric, which follows from the fact that $\fg_{1}$ is abelian.
\end{proof}
Now Theorem \ref{non_vmrt_} is a corollary of the following proposition.
\begin{prop}
	$\lambda$ induces an isomorhpism from $\fg_{-1}$ onto $\fa ut(\hat{\CC_{x}})^{(1)}$.
\end{prop}
\begin{proof}
	By \cite[Theorem 1.1.3]{prog_02} there is a natural inclusion: $\fa ut (\hat{\CC_{x}})^{1} \hookrightarrow (T_{x}X)^{\vee}$, whence  it suffices to show that $\lambda$ is injective. \par 
	Assume otherwise that $\lambda(d\phi_{y}(\beta))=0$ for some nonzero $\beta \in T_{y}X$. For any $\alpha \in T_{x}X$, denote by $\gamma_{\alpha}=[d\phi_{y}(\beta),d\phi_{x}(\alpha)] $. We now show that $\gamma_{\alpha}=l(\alpha) \, \text{Id}_{V}$ for some $l \in (T_{x}X)^{\vee}$. To do this it suffices to show that $G_{\gamma_{\alpha}}$ acts on $V$ by scalar multiplications. By Lemma 3.4 we have  $d\Phi_{\gamma_{\alpha}}(\gamma_{\alpha})=0$. From (3.3) it implies that $G_{\gamma_{\alpha}}$ acts trivally on $C^{+}(x)$, and consequently  on $X$ as $C^{+}(x)$ is open dense in $X$.  Thus for any nonzero vector $v \in \hat{X}$, $v$ is an eigenvector of the linear action of $G_{\gamma_{\alpha}}$ on $V$. For any $g \in G_{\gamma_{\alpha}}$ denote by $\{V_{g(c)}: c \in J_{g} \}$ the characteristic subspaces of the action of $g$ on $V$, then $X \subset  \cup_{c \in J_{g}} \BP V_{g(c)}$. As  $X$ is non-degenrate and irreducible, there exists some $c \in J_{g}$ such that $V_{g(c)}=V$, i.e., $g$ acts on $V$ by scalar multiplications.  \par 
	Now denote by $U_{\beta}=\{exp(d\phi_{y}(t\beta)): t \in \BC\} \subset GL(V)$ the 1-dimensional vector subgroup of I$m(\phi_{y})$ and $V_{l}=\{exp(d\phi_{x}(\alpha)): l(\alpha)=0\}$ the vector subgroup of Im$(\phi_{x})$.  Then by Corollary \ref{y_act} $x$ is fixed by $U_{\beta}$ as $x=[e_{0}]$ and $d \phi_{y}(\beta) \cdot e_{0}=0$. Thus $V_{l}.x$ is also fixed by $U_{\beta}$ as the action of $U_{\beta}$ and $V_{l}$ on $X$ commutes by our definition of $l$. On the other hand  as  $\beta \not=0$, $U_{\beta}$ acts freely on $C^{-}(y) \cong T_{y}X$. This implies that  $V_{l} \cdot x \subset C^{+}(x) \backslash C^{-}(y)$. We prove that this induces a contradiction:\par 
	If  $l=0$ then $V_{l}=Im(\rho_{x})$ and $V_{l} \cdot x=C^{+}(x)$. But $C^{+}(x) \cap C^{-}(y) $  is non-empty as they are both open dense in $X$.\par 
	If $l\not=0$ then $V_{l} \cdot x$ is a hyperplane in $C^{+}(x) \cong T_{x}X$. By the proof of Lemma \ref{lem}, $D_{y} :=X \backslash C^{-}(y)$ is an irreducible divisor of $X$. So as an open subset of $D_{y}$, $C^{+}(x) \backslash C^{-}(y)$ is an irreducible divisor of $C^{+}(x)$. This implies that $C^{+}(x) \backslash C^{-}(y)$ equals the hyperplane $V_{l} \cdot x$. On the other hand by Corollary \ref{y_act} we can write $C^{+}(x) \backslash C^{-}(y)=\{w \in T_{x}X: \eta(s_{r})(w)=0\}$ where $s_{r}$ is a nonzero section in $H^{0}(X,\CL)_{w_{r}}$. Thus $Bs(\BF_{x}^{r})=\{[w] \in \BP T_{x}X:  \eta(s_{r})(w)=0\}$ is a hyperplane in $\BP T_{x}X$. But then $\CC_{x}=Bs(\BF_{x}) \subset Bs(\BF_{x}^{r})$ is linear degenerate in $\BP T_{x}X$, contradicting  Proposition \ref{vmrt_nond}.
\end{proof}
\subsection{Projective subvarieties with nonzero prolongations}
Let us recall the classification result of projective subvarieties with nonzero prolongations by Fu and Hwang as follows.
\begin{thm}\label{non_zero1} \cite[Main Theorem]{nonzero_prog} and \cite[Theorem 7.13]{special_21} 
	Let $S \subset \BP V$ be an irreducible nonsingular nondegenerate variety such that $\fa ut(\hat{S})^{1} \not =0$. Then $S \subset \BP V$ is projectively equivalent to one of the followings:\par 
	(1) The VMRT of an IHSS.\par 
	(2) The VMRT of a symplectic Grassmannian. \par 
	(3) A nonsingular linear section of  $Gr(2,5) \subset \BP^{9}$ of codimension $\leqslant 2$. \par 
	(4) A nonsingular $\BP^{4}$-general linear section of $\BS_{5} \subset \BP^{15}$ of codimension $\leqslant 3$.\par 
	(5) Biregular projections of (1) and (2) with nonzero prolongations, which are completely described in Section 4 of \cite{nonzero_prog} .
\end{thm}
\begin{remark}
	As noted in \cite[Proposition 2.11]{special_21}, all nonsingular sections of $Gr(2, 5) \subset \BP^{9}$ with codimension $s \leqslant 3$ are projectively equivalent.
\end{remark}
The main result of this subsection is the following result based on Theorem 3.7.
\begin{prop}\label{non_zero2}
	Let $S \subset \BP V$ be one of the projective subvarieties in Theorem 3.7 (2)(3)(4)(5), then:
	\begin{equation}
		dim(\fa ut(\hat{S})^{(1)}) < dim(V).
	\end{equation}
\end{prop}
We will prove this proposition  case by case based on Theorem 3.7. \par 

\subsubsection{Case (2) and (5)}  In these cases, the prolongation of $\fa ut (\hat{S})$  was explicitly formulated  in \cite{nonzero_prog}. First we consider Case (2) and the case of biregular projections of (2):
\begin{lem}
	Let $W$ and $Q$ be vector spaces of dimensions $k \geqslant 2$ and $m$ respectively. Set $L=Sym^{2}(Q) \subset V=Sym^{2}(W \oplus Q)$ and $U=V/L$. For $\phi \in Sym^{2}(W \oplus Q)$ denote by $\phi^{\#} \in Hom(W^{\vee} \oplus  Q^{\vee},W \oplus Q)$ the corresponding homomorphism via the natrual inclusion $Sym^{2}(W \oplus Q) \subset Hom((W^{\vee} \oplus  Q^{\vee},W \oplus  Q)$. For $L_{2} \subset U$, let $\text{Im}(L_{2})$ be the linear space of $\{\text{Im}(\phi^{\#})  : \overline{\phi}  \in L_{2}\}$. Define $\text{Im}_{W}(L_{2})=P_{Q}(Im(L_{2})) \subset W$, where $P_{Q}: W \oplus Q \rightarrow W$ is the projection to the first factor, then:\par 
	(i) Denote by $p_{L}: \BP V \dashrightarrow \BP(V/L)$ the projection from $\BP L$. Let $v_{2}: \BP (W \oplus Q) \rightarrow \BP(Sym^{2}(W \oplus Q))$ be the second Veronese  embedding, $Z$ the proper image of $Im(v_{2})$. Then $Z \subset \BP V/L=\BP U$ is isomorphic to the VMRT of the sympletic Grassmannian $Gr_{w}(k,\Sigma)$ at a general point and $\fa u t(\hat{Z})^{(1)} \cong Sym^{2}(W^{\vee})$.\par 
	(ii) If $Z \cap \BP L_{2} =\emptyset$, then $\fa ut(\widehat{p_{L_{2}}(Z)})^{(1)} \cong Sym^{2}(W/Im_{W}(L_{2}))^{\vee}$. \par 
	(iii) $dim(\fa u t(\hat{Z})^{(1)} ) < dim(V/L)$. Let $L_{2} \subset U$ be as in (ii), if $\fa ut(\widehat{p_{L_{2}}(Z)})^{(1)}  \not=0$, then:
	\begin{equation*}
		dim(\fa ut(\widehat{p_{L_{2}}(Z)})^{(1)}) < dim(U/L_{2}).
	\end{equation*}
\end{lem}
\begin{proof}
	(i) and (ii) are from \cite[Proposition 4.18]{nonzero_prog}. \par 
	Under the identification: $V=Sym^{2}(W \oplus Q) \subset Hom(W^{\vee} \oplus Q^{\vee}, W \oplus Q)$, we write $U=V/L=Sym^{2}(W) \oplus Hom(Q^{\vee},W)$, where we identify $Sym^{2}(W)$ inside $ Hom(W^{\vee},W)$.	Thus $dim(V/L) >dim(\fa u t(\widehat{Z})^{(1)} ) $ as $Hom(Q^{\vee},W) \not=0$. Take a basis of $Im_{W}(L_{2})$ to be $e_{1},...,e_{t}$ and extend it to a basis of $W$: $e_{1},...,e_{t},e_{t+1},...,e_{k}$. Then by the definition of $Im_{W}(L_{2})$, we have:
	\begin{align*}
		L_{2} \subset &\{(\phi,\eta) \in Sym^{2}(W) \oplus Hom(Q^{\vee},W) \, | \,  Im(\phi^{\#}) \subset Im_{W}(L_{2}) \,\, and \,\, Im(\eta)  \subset Im_{W}(L_{2}) \}\\
		&\cong Sym^{2}(Im_{W}(L_{2})) \oplus Hom(Q^{\vee},Im_{W}(L_{2})),
	\end{align*}
	whence $dim(L_{2}) \leqslant  \frac{t(t+1)}{2}+mt$. Then:
	\begin{align*}
		dim(U/L_{2})-dim(\fa ut(\widehat{p_{L_{2}}(Z)})^{(1)})&	 \geqslant \frac{k(k+1)}{2}+km-\frac{t(t+1)}{2}-tm-\frac{(k-t)(k-t+1)}{2}\\
		&=(m+t)(k-t) > 0,
	\end{align*} 
	where $k-t >0$ as $\fa ut(\widehat{p_{L_{2}}(Z)})^{(1)} \not=0$.
\end{proof}
By \cite[Main Theorem (C)]{nonzero_prog}, the other cases of (5) are biregular projections of the  VMRT of $Gr(a,a+b), \BS_{n}$ and $ Lag(n,2n)$ respectively. We prove these cases  by the following three lemmas.
\begin{lem}
	Let $A$ and $B$ be vector spaces with $a=dim(A) \geqslant b =dim(B) \geqslant 3$. Let $V=Hom(A,B)$. For a subspace $L \subset V$, set $\text{Im}(L)=\{\text{Im}(\phi) \subset B: \phi \in L\}$. $Ker(L)=\bigcap\limits_{\phi \in L} Ker(\phi)$. Then: \par 
	(i) $S=\{[\phi] \in \BP V: rank(\phi) \leqslant 1\} \subset \BP V$ is projectively isomorphic to the VMRT of $Gr(a,a+b)$.\par 
	(iii) Let $L \subset V$ such that $L \cap Sec(S) = \emptyset$, then $	\fa ut (\widehat{p_{L}(S)})^{(1)} \cong Hom(B/Im(L),Ker(L))$ \par
	
	(iii) Let $L \subset V$ be as in (ii), if $	\fa ut (\widehat{p_{L}(S)})^{(1)}  \not =0$  then  $dim(	\fa ut (\widehat{p_{L}(S)})^{(1)} ) < dim(V/L)$.
\end{lem}
\begin{proof}
	(i) and (ii) are from \cite[Proposition 4.10]{nonzero_prog}. For (iii), denote $s=dim(Ker(L))$ and $t=dim(Im(L))$, by the definition  we have
	\begin{equation*}
		L \subset \{\phi \in Hom(A,B): \phi|_{Ker(L)}=0, Im(\phi) \subset Im(L)\}=Hom(A/Ker(L),Im(L)),
	\end{equation*}
	implying $dim(L) \leqslant (a-s)t$. Thus
	\begin{equation*}
		dim(V/L)-dim(\fa ut (\widehat{p_{L}(S)})^{(1)})=ab-dim(L)-(b-t)s \geqslant ab-(a-s)t-(b-t)s=(a-t)(b-t)+st >0,	
	\end{equation*}
	where $s < a$ and $t<b$ as $	\fa ut (\widehat{p_{L}(S)})^{(1)} \not=0$. 
\end{proof}
\begin{lem}
	Let $W$ be a vector space of dimension $n \geqslant 6$. $V=\wedge^{2} W$. For each $\phi \in \wedge^{2} W$, denote by $\phi^{\#} \in Hom(W^{\vee},W)$ via the inclusion $\wedge^{2}W \subset W \otimes W=Hom(W^{\vee},W)$. For a subspace $L \subset V$, define $Im(L) \subset W$ as the linear span of $\{Im(\phi^{\#}) \subset W, \phi \in L\}$. Then:\par 
	(i)   $S=\{[\phi] \in V: rk(\phi) \leqslant 2\} \subset \BP V$ is isomorphic to the  VMRT of $\BS_{n}$.\par 
	(ii) If $L \subset V$ such that $\BP L \cap Sec(S)=\emptyset$, then $\fa ut(\widehat{p_{L}(S)})^{(1)} \cong \wedge^{2}(W/Im(L))^{\vee}$.\par 
	(iii) Let $L \subset V$ be as in (ii), if  $\fa ut(\widehat{p_{L}(S)})^{(1)} \not =0$, then  $dim(\widehat{p_{L}(S)})^{(1)}) < dim(V/L)$.
\end{lem}
\begin{proof}
	(i) and (ii) are from \cite[Proposition 4.11]{nonzero_prog}. For (iii)  take a basis of $Im(L)$ to be $e_{1},...,e_{t}$ and extend it to a basis of $W$: $e_{1},...,e_{t},e_{t+1},...,e_{n}$. Denote the dual basis of $W^{\vee}$ to be $f_{1},...,f_{n}$ such that $f_{i}(e_{j})=\delta_{i,j}$ for any $1 \leqslant  i,j \leqslant n$. Identify $Hom(W^{\vee},W)$ with $M_{n\times n}(\BC)$ through:
	\begin{align}
		Hom(W^{\vee},W) &\longrightarrow M_{n \times n}(\BC) \\
		\CA     &\longrightarrow A=(a_{ij}:  1 \leqslant  i,j \leqslant n) \nonumber 
	\end{align}
	such that $\CA(f_{i})=\sum_{j=1}^{n} a_{ij}e_{j}$. Then $V$ corresponds to all skew-symmetric matrices. Now 
	\begin{align*}
		L \subset \{\CA \in V: Im(\CA) \subset Im(L)\}=\{\CA \in V: a_{ij}=0 \,\,\, \text{if \,\,\,} i \geqslant r+1  \text{\,\,\,or\,\,\,} j \geqslant r+1\},
	\end{align*}
	thus $dim(L) \leqslant \frac{t(t-1)}{2}$. Then:
	\begin{align*}
		dim(V/L)-dim(\fa ut(\widehat{p_{L}(S)})^{(1)})&=\frac{n(n-1)}{2}-dim(L)-\frac{(n-t)(n-t-1)}{2}  \\
		& \geqslant \frac{n(n-1)}{2}-\frac{t(t-1)}{2}-\frac{(n-t)(n-t-1)}{2}=t(n-t) > 0,
	\end{align*}
	where $t>0$ as $L \not=0$ and $t < n$ as $\fa ut(\widehat{p_{L}(S)})^{(1)} \not=0$.
\end{proof}

\begin{lem}
	Let $W$ be a vector space of dimension $n \geqslant 3$. $V=Sym^{2} W$. For each $ \phi \in Sym^{2} W$, denot by $\phi^{\#} \in Hom(W^{\vee},W)$ via the inclusion $Sym^{2}W \subset W \otimes W=Hom(W^{\vee},W)$. For a subspace $L \subset V$, define $Im(L) \subset W$ as the linear span of $\{Im(\phi^{\#}) \subset W, \phi \in L\}$. Then:\par 
	(i)   $S=\{[\phi] \in V: rk(\phi) \leqslant 1\} \subset \BP V$ is isomorphic to the  VMRT of $Lag(n,2n)$.\par 
	(ii) If $L \subset V$ such that $\BP L \cap Sec(S) =\emptyset$, then $\fa ut(\widehat{p_{L}(S)})^{(1)} \cong Sym^{2}(W/Im(L))^{\vee}$.\par 
	(iii) Let $L \subset V$ be as in (ii), if  $\fa ut(\widehat{p_{L}(S)})^{(1)} \not =0$, then  $dim(\widehat{p_{L}(S)})^{(1)}) < dim(V/L)$.
\end{lem}
\begin{proof}
	(i) and (ii) are from \cite[Proposition 4.12]{nonzero_prog} . For (iii), as in Lemma 3.11 we take a basis of $Im(L)$ to be $e_{1},...,e_{r}$ and extend it to a basis of $W$ to be $e_{1},...,e_{r},e_{r+1},...,e_{n}$. Denote the dual basis to be $f_{1},...,f_{n}$.  Keep the  identification  (3.5) then  $V$ corresponds to all symmetric matrices and we  have
	\begin{align*}
		L \subset \{\CA \in V: Im(\CA) \subset Im(L)\}=\{\CA \in V: a_{ij}=0 \,\,\, \text{if \,\,\,} i \geqslant r+1  \text{\,\,\,or\,\,\,} j \geqslant r+1\},
	\end{align*}
	implying $dim(L) \leqslant \frac{r(r+1)}{2}$ and thus 
	\begin{align*}
		dim(V/L)-dim(\fa ut(\widehat{p_{L}(S)})^{(1)})&=\frac{n(n+1)}{2}-dim(L)-\frac{(n-r)(n-r+1)}{2}  \\
		& \geqslant \frac{n(n+1)}{2}-\frac{r(r+1)}{2}-\frac{(n-r)(n-r+1)}{2}=r(n-r) > 0,
	\end{align*}
	where $r>0$ as $L \not=0$ and $r < n$ as $\fa ut(\widehat{p_{L}(S)})^{(1)} \not=0$.
\end{proof}
\subsubsection{Case (3)} Let $X=Gr(2,5) \subset \BP^{9}$. For each $k=1,2$, denote by $X_{k} \subset \BP^{9-k}$ the nonsingular linear section of codimension $k$. Then Case $X_{1}$ follows from \cite[Section 3.4]{nonzero_prog}  and Case $X_{2}$ follows from \cite[Lemma 4.6]{Fano_complete}.
\subsubsection{Case (4)} Let $S=\BS_{5} \subset \BP^{15}$. For each $k=1,2,3$, denote by $S_{k} \subset \BP^{15-k}$ the nonsingular $\BP^{4}$-general linear section of codimension $k$ as described in \cite[Proposition 2.12]{special_21}. \par 
(i) Case $S_{1}$ follows from \cite[Section 3.3]{nonzero_prog}.\par 
(ii)
By \cite[Proposition 7.6]{special_21}  $S_{3}$ is quadratically symmetric. The VMRT of $S_{3}$ at a general point  is a nonsingular linear section of $Gr(2,5) \subset \BP^{9}$ of codimension 3, which has zero prolongations  by Theorem \ref{non_zero1}. Then by the proof of \cite[Theorem 6.15]{nonzero_prog} we conclude that $\fa ut(\hat{S_{3}}) \cong \BC$.\par 
(iii)  To prove Case $S_{2}$ we recall the following characterization of $S_{2}$ proved  by Kuznetsov.
\begin{thm}\cite[Proposition 6.1 and  Lemma 6.7]{section_spin10}
	Let $S_{K} \subset S$  be a nonsingular linear section of $S$ of codimension 2, then the followings are equivalent:\par 
	(a) $S_{K}$ is projectively equivalent to $S_{2}$;\par 
	(b)  The Hilbert space $F_{4}(S_{K})$ of linear 4-spaces on $S_{K}$ is non-empty;\par 
	(c) There exists a line $L$ in $S_{K}$ such that
	\begin{equation}
		\CN_{L/S_{K}} \cong \CO_{\CL}(-2) \oplus \CO_{L}(1)^{\oplus 6} 
	\end{equation}
	Moreover, such line is unique and is equal to the intersection of all 4-spaces on $S_{K}$.
\end{thm}
Now assume otherwise that $dim(\fa ut(\hat{S_{2}})^{(1)})=dim(V)$. Take $L$ as the line defined in Theorem 3.14. For any point $x=[\hat{x}] \in L$, take a linear function $l \in V^{\vee}$ such that $l(\hat{x}) \not=0$. Then by \cite[Proposition 2.3.1]{prog_02} there exists a unique $\CA \in \fa ut(\hat{S_{2}})^{(1)}$ such that $\CA_{\alpha, \alpha}=l(\alpha) \alpha$ for any  $\alpha \in \hat{S_{2}}$. Moreover denote by $P_{\alpha}=T_{\alpha}(\hat{S_{2}})$ the tangent space of $\hat{S_{2}}$ at $\alpha$, then we have:
\begin{equation}
	2 \CA_{\alpha , \beta}=l(\alpha) \beta + l(\beta) \alpha,
\end{equation}
for any $\alpha \in \hat{S_{2}}$ and any $\beta \in P_{\alpha}$. 
Denote by $s$ the semisimple part of $\CA_{\hat{x}}$. By the proof of \cite[Theorem 1.1.3]{prog_02}, the one parameter subgroup $\{exp(2ts): t \in \BC\}$ induces a $\BC^{*}$-action on $S_{2}$ which is of Euler type at $x$.  We shall deduce the contradiction from the $\BC^{*}$-action.  \par 
First we claim that there are exactly three different weight subspaces of the $\BC^{*}$-action on $V$. In fact by $\CA_{\hat{x}}(V) \subset P_{\hat{x}}$ and by (3.8), the linear action of $\BC^{*}$ on $V$ has at most three different weight subspaces. On the other hand denote by $\CL=\CO_{\BP V}(1)|_{S_{2}}$. As $S_{2} $ is linear normal in $\BP V$,  $(S_{2}, \CL,x)$ satisfies the conditions in Lemma 2.7. Thus  by Corollary 2.8  it has exactly three weighted subspaces. 
Under the setting of Section 2 we have $r=2$. Denote by $W=H^{0}(S_{2},\CL)^{\vee}_{-1}$, $U=H^{0}(S_{2},\CL)^{\vee}_{w_{2}}$ and $f: S_{2} \rightarrow \BP V$ the projective embedding. Then we have: $dim(V)=14$, $  dim(W)=dim(T_{x}S_{2})=8$ and $dim(U)=dim(V)-1-dim(W)=5$. We now  check that the $\BC^{*}$-action on $S_{2}$ satisfies the following two properties. \par 
(a) There are exactly three irreducible components of $X^{\BC^{*}}$:  the isolated source $\{x\}$, the unique component $Y_{1}$ contained in $\BP W$ and the unique component $Y_{2}$  contained in $\BP U$.   First note that for any $Y \in \CY$, $Y \subset \BP V_{k}$ for some k.  If $Y \subset \BP W$, then  a $\BC^{*}$-orbit whose sink lies in $Y$ has its source equal to $x$. Thus by Proposition 2.3 such $Y$ is unique, and $C^{-}(Y)$ is a line bundle over $Y$. If $Y \subset \BP U$ then we see that $v^{+}(Y)=0$ whence $Y$ is the unique sink of the $\BC^{*}$-action. \par 
(b) We have  $dim(Y_{1}) > 0$. Otherwise assume that  $Y_{1}=\{y\}$ is a single point. If  $dim(Y_{2}) = 0$ then by Proposition 2.3 and (a),  we have $v^{+}(Y_{1})=v^{-}(Y_{1})=1$ and thus  $dim(S_{2})=dim(T_{y}S_{2})= v^{+}(Y_{1})+v^{-}(Y_{1})=2 < 8$, which is a contradiction. If $dim(Y_{2}) > 0$ then we have $D_{x}=S_{2} \backslash C^{+}(x)=C^{+}(y) \cup Y_{2}$. This implies that $Y_{2}$, as a divisor of $D_{x}$, is of dimension 6, contradicting the fact that $Y_{2} \subset \BP U \cong \BP^{4}$.\par 
Now as $L$ is the intersection of all 4-spaces in $X_{2}$, it is  $\BC^{*}$-invariant. Moreover as $x \in L$ and the action is of Euler type at $x$, we conclude that $L$ is a non-trival $\BC^{*}$-orbit closure with source $x=[e_{0}]$.  Denote the orbit to be $\BC^{*} \cdot f(w)=\{[e_{0} + z \Pi_{1}(w) + z^{-w_{2}}\Pi_{2}(w,w)]: z \in \BC^{*}\}$ for some nonzero $w \in T_{x}S_{2}$ and denote by $y$  the sink of the orbit. Then we must have $\Pi_{2}(w,w)=0$ and $y \in Y_{1}$. Otherwise the sink of the orbit would be $[\Pi_{2}(w,w)] \in \BP U$. Then the line $L$ will be contained in $  \BP (\BC e_{0}  \oplus U)$, contradicting the fact that $\Pi_{1}(w) \not =0$ as $\Pi_{1}$ is injective by Lemma 2.7 (3). Now take any point $y' \in Y_{1}$, dentoe by $L_{y'}$ the unique non-trival $\BC^{*}$-orbit closure with sink $y'$ and source $x$. Then $L_{y'}$ is exactly the line connecting $x$ and $y'$. By \cite[Lemma 2.16]{small_bandwidth} and by the proof of \cite[Proposition 2.17]{small_bandwidth}, the splitting type of $T_{S_{2}}|_{L_{y'}}$ is determined by the weights of the isotropy action of $\BC^{*}$ on $T_{y'}S_{2}$. As $Y_{1}$ is irreducible, the weights of $\BC^{*}$ on $T_{y'}S_{2}$ remain  invariant as $y'$ varies in $Y_{1}$. This implies that the splitting type of $T_{S_{2}}|_{L_{y'}}$ also remains invariant, contradicting the uniqueness of $L$ as $dim(Y_{1}) > 0$. Thus we conclude that $dim(\fa ut(\hat{S_{2}})^{(1)}) < dim(V)$. This completes the proof of Proposition 3.9.

\section{Proof of main result}
In this section we will prove Theorem \ref{main_thm}. Let's first recall the Cartan-Fubini type extension theorem proved by Hwang and Mok \cite{Cartan_fubini}. We will use the following version taken from  \cite[Theorem 6.8]{nonzero_prog}
\begin{thm}
	Let $X_{1}$ and $X_{2}$ be two Fano manifolds of Picard number 1, different from projective spaces. Let $\CK_{1}$ and $\CK_{2}$ be families of minimal rational curves on $X_{1}$ and $X_{2}$ respectively. Assume that for a general point $x \in X_{1}$, the VMRT $\CC_{x} \subset \BP T_{x}(X_{1})$ is irreducible and nonsingular. Let $U_{1} \subset X_{1}$ and $U_{2}  \subset X_{2}$ be connected analytical open subsets. Suppose that there exists a biholomorphic map $\phi: U_{1} \rightarrow  U_{2}$ such that for a general point $x \in U_{1}$, the differential $d\phi_{x} : \BP T_{x}(U_{1}) \rightarrow  \BP T_{\phi(x)}(U_{2})$ sends $\CC_{x}$ isomorphically to $\CC_{\phi(x)}$. Then there exists a biregular morphism $\Phi: X_{1} \rightarrow  X_{2}$ such that $\phi = \Phi|U_{1}$.
\end{thm}
The extension theorem enables us to characterize an Euler-symmetric variety by its VMRT. We present a $\BC^{*}$-equivariant version for our convenience.
\begin{cor}\label{c_f}
	Let $X,X'$ be two Fano manifolds of Picard number 1, and $\CK,\CK'$ are families of minimal rational curves on $X,X'$ respectively. $\CC \subset \BP T(X)$ and $\CC' \subset \BP T(X')$ the associated VMRT sturctures on $X$ and $X'$. Assume that for  a general point  $x$ on $X$ and for a general point $x'$ on $X'$ there are $\BC^{*}$-actions on $X$ and $X'$  such that the actions are of Euler type at $x$ and $x'$ respectively. If  $\CC_{x}$ is projectively isomorphic to $\CC'_{x}$, then there is a $\BC^{*}$-equivariant isomorphism $\Phi: X \rightarrow X'$ maps $x$ to $x'$.
\end{cor}
\begin{proof}
	Assume  the projective isomorphism between $\CC_{x} \subset \BP T_{x}X$ and $\CC'_{x'} \subset \BP T_{x'}X'$ is given by the linear isomorphism $\phi: T_{x}X \rightarrow T_{x'}X'$. Then the identifications $C^{+}(x) \cong T_{x}X$ and $C^{+}(x') \cong T_{x'}X'$ induce an isomorphism $\overline{\phi}: C^{+}(x) \rightarrow C^{+}(x')$.   By Proposition 3.3, their VMRTs at general points are both irreducible and nonsingular.  To extend $\overline{\phi}$, it suffices to check the differential map of  $\overline{\phi}$ preserves the VMRT at a general point. As $x$ and $x'$ are both taken as general points, by Proposition 2.10 (3)  the action of $T_{x}X$ on $C^{+}(x)$ and the action of $T_{x'}X'$ on $C^{+}(x')$ can be extended to an action on $X$ and $X'$ respectively. This shows the VMRT structures over $C^{+}(x)$ and $C^{+}(x')$ are both locally flat. Thus by our definition of $\overline{\phi}$,  for any point $y \in C^{+}(x)$, $d\overline{\phi}_{y}$ must map $\CC_{y}$ isomorphically onto $\CC'_{\overline{\phi}(y)}$. This yields the existence of $\Phi$. Finally by our assumption $\overline{\phi}$ is $\BC^{*}$-equivariant, thus $\Phi$ is also $\BC^{*}$-equivariant as $C^{+}(x)$ is open dense.
\end{proof}

Now let $X=G/P=\CD(I)$ be a rational homogenous space of Picard number one defined by a simple algebraic group $G$ and a parabolic subgroup $P$ given by  $I=\{\alpha \}$ for a simple positive root $\alpha \in \Delta$. We review some facts about $\BC^{*}$-actions on rational homogenous spaces, most of which are taken from \cite{ORSW_2021}.\par 
Up to composing with a character, a $\BC^{*}$-action on $X$ is  given by  a cocharacter $\sigma: \BG_{m} \rightarrow T$ and the left multiplication of the maximal torus $T$ on $G$.  For a simple root $\beta$, we denote the cocharacter $\sigma_{\beta}$ by assigning  $\sigma_{\beta}(\eta)=\delta_{\eta,\beta}$ for any $\eta \in \Delta$. For a cocharacter $\sigma$, denote $\fg_{k}=\bigoplus\limits_{\theta \in \Phi: \sigma(\theta)=k} \fg_{\theta}$.
\begin{prop}
	Let $X=G/P=\CD(I)$ be a rational homogenous space of Picard number one, then:\par 
	(i) A $\BC^{*}$-action on $X$ given by a cocharacter $\sigma \in X_{*}(T)$ is equalized if and only if  the grading of $\fg$ is short, i.e., $\fg=\fg_{1} \oplus \fg_{0} \oplus  \fg_{-1}$. In this case, up to a conjuagation we can assume $\sigma=\sigma_{\beta}$ for some $\beta \in \Delta$. \par 
	(ii)  A $\BC^{*}$-action on $\CD(I)$ given by the cocharacter $\sigma_{\beta}$ is equalized with an isolated sink if and only if $\alpha=\beta$ and  $\sigma_{\alpha}$ defines a short grading of $\fg$, equivalently $X$ is an IHSS and $\sigma=\sigma_{\alpha}$. In this case,  the $\BC^{*}$-action is equalized of weight -1 at the sink $x'=eP$, and  it has an isolated source if and only if  $X=\CD(I)$ is an IHSS of tube type, where the isolated source is $y'=\dot{w_{o}}x'$.\par 
	(iii)  Assme that $X=\CD(I)$ is an IHSS, with the $\BC^{*}$-action given by $\sigma_{\alpha}$. The maximal dimensional Bruhat cell  is $C^{-}(x')=R_{u}(P^{-}) \cdot x'$ and $X$ is  an equivariant compactification of $R_{u}(P^{-})$. If $X$ is of tube type, then  $C^{+}(y')=R_{u}(P) \cdot y'$ and $X$ an equivariant compactification of $R_{u}(P)$.
\end{prop}
Now assume $X=\CD(I)$ to be an IHSS where $I=\{\alpha\}$ for some simple root $\alpha$  and consider the $\BC^{*}$-action on $X$ given by the cocharacter $-\sigma_{\alpha}$. Then the action is of Euler type at $x=eP$ and its inverse action is of Euler type at $y=\dot{w_{0}} \cdot x$. By (iii)  the $\BC^{*}$- action satisfies the condition of Proposition 2.4 (ii). \par 
We are ready to prove our main result.
\begin{proof}[Proof of Theorem \ref{main_thm}]
	If $X=\CD(I)$ is an IHSS of tube type,  denote the $\BC^{*}$-action on $X$ by $-\sigma_{\alpha}$. Then  $X$ is the equivariant compactification of  $R_{u}(P^{-})$ and $R_{u}(P)$ respectively, where  $x=eP$ is fixed by $R_{u}(P)$ and $y=\dot{w_{0}}x$ is fixed by $R_{u}(P^{-})$. Consider the morphism  $R_{u}(P) \times R_{u}(P^{-}) \rightarrow X \times X: (u,v) \rightarrow (uv \cdot x,u \cdot y)$, its image is a dense and constructible, whence contains a dense open subset. Moreover for any element $(u \cdot y,uv \cdot x)$ in the image
	we can define a $\BC^{*}$-action on $X$ by:
	\begin{align*}
		\BC^{*} \times X &\longrightarrow X \\
		(\,\,t \,\,\,,\,\,x\,\,) &\longrightarrow u v \, t \, (uv)^{-1} \cdot x,
	\end{align*}
	such that the $\BC^{*}$-action is of Euler type at the source $\tilde{y}=uv \cdot y=u \cdot y$ and its inverse action is of Euler type at $\tilde{x}=uv \cdot x$. 
	\par 
	Conversely assume that for a general pair of points $x,y$ on $X$, there is a $\BC^{*}$-action which is of Euler type at $x$ and the inverse action is of  Euler type at $y$. Take $\CK$  to be a family of minimal ratioinal curves on $X$, and let $\CC \subset \BP T(X)$ be the VMRT structure. By Theorem  \ref{non_vmrt_}, for the projective embedding $\CC_{x} \subset \BP T_{x}X$ we have $dim(\fa ut(\hat{\CC_{x}})^{(1)})=dim(T_{x}X)$. Thus from Theorem \ref{non_zero1}  and Proposition \ref{non_zero2}, $\CC_{x}$ is projectively isomorphic to the VMRT of an IHSS. Denote the IHSS by $X'=\CD(I)$ and consider the $\BC^{*}$-action on $X'$ given by $-\sigma_{\alpha}$ with the isolated source $x'$. Then $(X,x, X',x')$ satisfies the condition  of Corollary \ref{c_f} and hence $X$ is $\BC^{*}$-equivariant isomorphic to $X'$. As the $\BC^{*}$-action on $X$ has isolated sink and source,  $X'$ must be an IHSS of tube type by Proposition 4.3 (ii).
\end{proof}
\section*{Acknowledgements}
I am very grateful to my advisor Baohua Fu for sharing valuable ideas on Euler-symmetric varieties with me. I am also indebeted to him for many helpful  suggestions during the preparation of the paper.  I would like to thank Cong Ding  for reading the first draft of this paper and for useful discussions and suggestions. I am grateful to Zhijun Luo for helpful discussions.
\bibliographystyle{alpha}

\end{document}